\documentclass[12pt,a4paper]{article}

\usepackage{latexsym}
\usepackage{amsmath}
\usepackage{amssymb}
\usepackage{arydshln}

\usepackage{mathrsfs}

\usepackage{cite}
\usepackage{stmaryrd}
\usepackage{enumerate}

\usepackage{hyperref}

\usepackage{color}
\usepackage{lineno}
\usepackage{graphicx}
\usepackage{ae}
\usepackage{amsmath}
\usepackage{amssymb}
\usepackage{latexsym}
\usepackage{url}
\usepackage{epsfig}
\usepackage{mathrsfs}
\usepackage{amsfonts}
\usepackage{amsthm}

\usepackage{float}
\usepackage{subfig}
\usepackage{tikz}
\usetikzlibrary{positioning}

\newtheorem{theorem}{Theorem}[section]

\newtheorem{lemma}[theorem]{Lemma}

\newtheorem{cor}[theorem]{Corollary}

\newtheorem{example}{Example}

\theoremstyle{definition}

\numberwithin{equation}{section} %%��ʽ���Ÿ����½�
\allowdisplaybreaks    %%��ʽ��ҳչʾ

\def\qed{\hfill$\Box$\vspace{12pt}}

\long\def\delete#1{}

\usetikzlibrary{decorations.markings}

\tikzstyle{vertex}=[circle, draw, inner sep=0pt, minimum size=6pt]
\tikzstyle{directed}=[postaction={decorate,
	decoration={markings,mark=at position 0.3 with {\arrow{stealth}}}
}]%%
\usepackage{xcolor}
\usepackage[normalem]{ulem}

\begin{document}
\title {Pair state transfer in tensor product and double cover}

\author{Ming Jiang$^{a,b}$,~Xiaogang Liu$^{a,b,c,}$\thanks{Supported by the National Natural Science Foundation of China (No. 12371358) and the Guangdong Basic and Applied
Basic Research Foundation (No. 2023A1515010986).}~$^,$\thanks{ Corresponding author. Email addresses: mjiang@mail.nwpu.edu.cn, xiaogliu@nwpu.edu.cn,
wj66@mail.nwpu.edu.cn}~,~Jing Wang$^{a,b}$
\\[2mm]
{\small $^a$School of Mathematics and Statistics,}\\[-0.8ex]
{\small Northwestern Polytechnical University, Xi'an, Shaanxi 710072, P.R.~China}\\
{\small $^b$Research \& Development Institute of Northwestern Polytechnical University in Shenzhen,}\\[-0.8ex]
{\small Shenzhen, Guandong 518063, P.R. China}\\
{\small $^c$Xi'an-Budapest Joint Research Center for Combinatorics,}\\[-0.8ex]
{\small Northwestern Polytechnical University, Xi'an, Shaanxi 710129, P.R. China}\\[2mm]
{\small \emph{Dedicated to Professor Eddie Cheng on the occasion of his 60th birthday.}}
}
\date{}

\openup 0.5\jot
\maketitle

\begin{abstract}
Quantum state transfer, first introduced by Bose in 2003, is an important physical phenomenon in quantum networks, which plays a vital role in quantum communication and quantum computing. In 2004, Christandl et al. proposed the concept of perfect state transfer on graphs by modeling the quantum network using graphs, and unveiled the feasibility of applying graph theory to quantum state transfer. In 2018, Chen and Godsil proposed the definition of Laplacian perfect pair state transfer on graphs, which is a brilliant generalization of perfect state transfer. In this paper, we investigate the existence of Laplacian perfect pair state transfer in tensor product and double cover of two regular graphs, respectively, and reveal fundamental connections between perfect state transfer and Laplacian perfect pair state transfer. We give necessary and sufficient conditions for the tensor product of two regular graphs to admit Laplacian perfect pair state transfer, where one of the two regular graphs admits perfect state transfer or Laplacian perfect pair state transfer. Additionally, we characterize the existence of Laplacian perfect pair state transfer in the double cover of two regular graphs. By our results, a variety of families of graphs admitting Laplacian perfect pair state transfer can be constructed.

%Quantum state transfer, first introduced by Bose in 2003, is an important physical phenomenon in quantum networks, which plays a vital role in quantum communication and quantum computing. In 2004, Christandl et al. proposed the concept of perfect state transfer on graphs by modeling the quantum network using graphs, and unveiled the feasibility of applying graph theory to quantum state transfer. In 2018, Chen and Godsil proposed the definition of Laplacian perfect pair state transfer on graphs, which is a brilliant generalization of perfect state transfer.
\smallskip



\end{abstract}

%described the evolution of a time-dependent quantum system via continuous quantum walk on graphs and introduced the concept of perfect state transfer (short for PST) on graphs by modeling the quantum network by graphs.
%This work unveiled the feasibility of applying graph theory to quantum state transfer and attracted extensive research attention from scholars.
%There have been many results about the existence of PST on several families of graphs, implying that graphs admitting PST are rare.

%proposed the concept of perfect state transfer (short for PST) on graphs by modeling the quantum network by graphs, demonstrated the evolution of a time-dependent quantum system via continuous quantum walk on graphs, which unveiled the feasibility of applying graph theory to quantum state transfer and attracted extensive research attention from scholars.
%Many results about the existence of PST on several families of graphs were subsequently given, implying that the graphs admitting PST are rare.

\section{Introduction}
Quantum state transfer in quantum networks, first introduced by Bose in 2003 \cite{BOSE2}, is a very important research content for quantum communication protocols. It plays a vital role in supporting quantum communication \cite{chris1, Yin}, enabling distributed quantum computing architectures \cite{Main}, facilitating large-scale quantum network construction \cite{ZHU}, optimizing quantum algorithms \cite{Xu, Banks}, and so on. Achieving quantum state transfer with high fidelity is a core research priority in quantum information processing. In 2004, Christandl et al. \cite{chris1} innovatively proposed modeling the quantum network by graphs and demonstrated the evolution of a time-dependent quantum system via continuous quantum walk on graphs, which unveiled the feasibility of applying graph theory to quantum state transfer.

Let $G$ be a simple undirected graph with an associated Hermitian matrix $M_G$. The \emph{transition matrix} of continuous quantum walk on $G$ with respect to $M_G$ is defined as
\begin{equation*}
U_{M_{G}}(t)=\exp(-\mathrm{i} t M_{G}),~ t \in \mathbb{R},~\mathrm{i}=\sqrt{-1},
\end{equation*}
which simulates the evolution of the quantum system modeled by the graph $G$ over time. In different quantum spin models, the Hermitian matrix $M_G$ is usually taken as different matrices of $G$, such as the \emph{adjacency matrix} $A_G$, the \emph{Laplacian matrix} $L_G$, the \emph{signless Laplacian matrix} $Q_G$, etc. Let $|G|$ denote the number of vertices of $G$. The standard basis vector $\mathbf{e}_{u}^{|G|}\in \mathbb{C}^{\left|G\right|}$ indexed by the vertex $u$ in $G$ is called the \emph{vertex state} of $u$, which can be written as $\mathbf{e}_{u}$ whenever no confusion can arise. The graph $G$ is said to have \emph{perfect state transfer} (PST for short) between $u$ and $v$ at time $\tau$ if there exists a phase factor $\lambda\in \mathbb{C}$ with $\left|\lambda\right|=1$ satisfying
\begin{align*}
U_{A_{G}}(\tau)\mathbf{e}_{u}=\lambda \mathbf{e}_{v}.
\end{align*}
The existence of PST has been investigated in many families of graphs, including integral circulant graphs \cite{Pet}, cubelike graphs \cite{Cheung}, distance regular graphs \cite{distance} and Cayley graphs \cite{CaoCL20, CaoF21, Tan19,Tm19}. Analogously, replacing $A_{G}$ by $L_{G}$ in the definition of PST, we obtain the definition of \emph{Laplacian perfect state transfer} (LPST for short). In 2009, Bose et al. \cite{BOSE1} first gave the definition of LPST, and proved that the complete graph $K_n$ with one edge deleted admits LPST when $n$ is a multiple of $4$. For more results on LPST, please refer to \cite{Ack, Coutinho11,LiLZZ21, Alvi, LiuD, Kirk, LiY2, WJ}.

Although there are many results on state transfer, graphs admitting (Laplacian) perfect state transfer are still very rare. In 2018, Chen and Godsil further demonstrated quantum state transfer between two distinct pairs of vertices in graphs \cite{QCh18, ChG19}, which is a brilliant generalization of state transfer. Let $\{a, b\}$ be a pair of distinct vertices of $G$. The state of $\{a, b\}$ is formed by $\mathbf{e}_{a}-\mathbf{e}_{b}$, called the \emph{pair state} of $\{a, b\}$. In particular, if $a$ is adjacent to $b$, then $\mathbf{e}_{a}-\mathbf{e}_{b}$ is also called the \emph{edge state} of $\{a, b\}$. Given two distinct pairs of vertices $\{a, b\}$ and $\{c, d\}$, if there exists a time $\tau$ and a phase factor $\chi\in \mathbb{C}$ with $\left|\chi\right|=1$ such that
\begin{align*}
U_{M_{G}}(\tau)(\mathbf e_{c}-\mathbf e_{d})=\chi(\mathbf e_{a}-\mathbf e_{b}),
\end{align*}
then we say that $G$ has \emph{perfect pair state transfer} (Pair-PST for short) between $\mathbf e_{a}-\mathbf e_{b}$ and $\mathbf e_{c}-\mathbf e_{d}$ at time $\tau$ with respect to $M_{G}$, which is equivalent to
\begin{align*}
\left|\frac{1}{2}(\mathbf e_{c}-\mathbf e_{d})^\top U_{M_{G}}(\tau)(\mathbf e_{a}-\mathbf e_{b})\right|^{2}=1,
\end{align*}
where $\ast^\top$ denotes the transpose of $\ast$. In particular, if $\mathbf e_{a}-\mathbf e_{b}$ and $\mathbf e_{c}-\mathbf e_{d}$ are edge states, then Pair-PST is also called Edge-PST. Furthermore, $\mathbf e_{a}-\mathbf e_{b}$ is said to be \emph{periodic} with respect to $M_{G}$ at time $\tau$ if
\begin{align*}
\left|\frac{1}{2}(\mathbf e_{a}-\mathbf e_{b})^\top U_{M_{G}}(\tau)(\mathbf e_{a}-\mathbf e_{b})\right|^{2}=1.
\end{align*}

In 2018, Chen and Godsil first introduced the concept of \emph{Laplacian perfect pair state transfer} (Pair-LPST for short) by choosing the Hermitian matrix $M_{G}$ as $L_G$, and gave lots of foundational and useful results on Pair-LPST \cite{QCh18, ChG19}. In 2021, Cao investigated the existence of Edge-LPST in cubelike graphs, and proposed methods to obtain some classes of infinite graphs admitting Edge-LPST \cite{CAO}.
In the same year, Luo et al. revealed the relation between PST and Edge-LPST over Cayley graphs and gave necessary and sufficient conditions for Cayley graphs over dihedral groups to have Edge-LPST \cite{LCXC21}.
In 2022, Cao and Wan gave a characterization of abelian Cayley graphs having  Edge-LPST \cite{CAO2}.
In 2023, Wang, Liu and Wang gave sufficient conditions for the vertex corona to have or not have Pair-LPST and they also proposed the definition of \emph{Laplacian pretty good pair state transfer} in \cite {liu2023}.
In 2024, Kim et al. generalized pair state to \emph{s-pair state} of $\mathbf e_{a}+s  \mathbf e_{b}$, where $s\in \mathbb{C}\setminus \left \{  0\right \} $, and developed the theory of \emph{perfect $s$-pair state transfer}, where the Hermitian matrix is taken to be the adjacency, Laplacian or signless Laplacian matrix of the graph \cite{Kim}.
In 2025, Tao and Wang characterized the existence of Edge-LPST on Cayley graphs of semi-dihedral groups \cite{Tao}.
In the same year, Jiang, Liu and Wang investigated the existence of Pair-LPST in Q-graphs and gave sufficient conditions for Q-graphs not having Pair-LPST \cite{JM}.

In this paper, we consider the existence of  Pair-LPST  in the tensor product and the double cover of two regular graphs, respectively. In Section \ref{Tensor product}, we give necessary and sufficient conditions for the tensor product of two regular graphs $G$ and $H$ to have Pair-LPST when $H$ admits Pair-LPST or PST, respectively. As examples, we construct several families of tensor products having Pair-LPST by paths, cycles and complete graphs. In Section \ref{Double cover}, we give necessary and sufficient conditions for the double cover of two regular graphs to have  Pair-LPST. We also prove that the double cover of two complete graphs of the same order admits Pair-LPST as an example.

\section{Preliminaries}\label{Sec:pre}

Let $G$ be a graph with Laplacian matrix $L_G$. Denote the set of all distinct eigenvalues of $L_G$ by $\mathrm {Spec}_{L}(G)$. Note that $L_G$ is real symmetric and diagonalizable. It follows that $L_G$ has the spectral decomposition
\begin{equation}
\label{spect1}
L_G=\sum_{\theta_r\in \mathrm {Spec}_{L}(G)}\theta_rF_{\theta_r},
\end{equation}
where $F_{\theta_r}$ is the \emph{eigenprojector} corresponding to $\theta_r\in \mathrm {Spec}_{L}(G)$. Notice the orthogonality and idempotence of eigenprojectors. By (\ref{spect1}), we have
\begin{equation}\label{LSpecDec2}
U_{L_G}(t)=\sum_{k\geq 0}\dfrac{(-\mathrm{i})^{k}L_G^{k}t^{k}}{k!}=\sum_{\theta_r\in \mathrm {Spec}_{L}(G)}\exp(-\mathrm{i}t\theta_{r})F_{\theta_r}.
\end{equation}
Given a pair state $\mathbf e_{a}-\mathbf e_{b}$ of $G$, the \emph{Laplacian eigenvalue support} of $\mathbf{e}_{a}-\mathbf{e}_{b}$ is defined as
$$
\mathrm {supp}_{L_G}(\mathbf{e}_{a}-\mathbf{e}_{b})=\left\{ \theta\in \mathrm {Spec}_{L}(G) \mid F_\theta\mathbf(\mathbf{e}_{a}-\mathbf{e}_{b})\neq \mathbf{0} \right\}.
$$
Two distinct pair states $\mathbf e_{a}-\mathbf e_{b}$ and $\mathbf e_{c}-\mathbf e_{d}$ are said to be \emph{Laplacian strongly cospectral} if for any eigenvalue $\theta \in \mathrm {Spec}_{L}(G)$,
$$
F_\theta\mathbf(\mathbf{e}_{a}-\mathbf{e}_{b})=\pm F_\theta\mathbf(\mathbf{e}_{c}-\mathbf{e}_{d}).
$$

Define
$$
\Lambda^{+}_{ab,cd}=\left\{\theta\in \mathrm {supp}_{L_G}(\mathbf{e}_{a}-\mathbf{e}_{b}) \mid F_\theta\mathbf(\mathbf{e}_{a}-\mathbf{e}_{b})= +F_\theta\mathbf(\mathbf{e}_{c}-\mathbf{e}_{d})\right\},
$$
and
$$
\Lambda^{-}_{ab,cd}=\left\{\theta\in \mathrm {supp}_{L_G}(\mathbf{e}_{a}-\mathbf{e}_{b}) \mid F_\theta\mathbf(\mathbf{e}_{a}-\mathbf{e}_{b})= -F_\theta\mathbf(\mathbf{e}_{c}-\mathbf{e}_{d})\right\}.
$$
Then $\mathbf e_{a}-\mathbf e_{b}$ and $\mathbf e_{c}-\mathbf e_{d}$ are Laplacian strongly cospectral if and only if
$$
\mathrm{{supp}}_{L_G}(\mathbf{e}_{a}-\mathbf{e}_{b})=\Lambda^{+}_{ab,cd}\cup \Lambda^{-}_{ab,cd}.
$$

In particular, we say that $\mathbf e_{a}-\mathbf e_{b}$ is Laplacian strongly cospectral with itself. In this case, $\Lambda^{-}_{ab,ab}=\emptyset$ and $\mathrm{{supp}}_{L_G}(\mathbf{e}_{a}-\mathbf{e}_{b})=\Lambda^{+}_{ab,ab}$ .

Let $A_G$ be the adjacency matrix of $G$, and denote the set of all distinct eigenvalues of $A_G$ by $\mathrm {Spec}_{A}(G)$. Similar to \eqref{spect1} and \eqref{LSpecDec2}, $A_G$ has the spectral decomposition
\begin{equation*}
A_G=\sum_{\mu_r\in \mathrm {Spec}_{A}(G)}\mu_rE_{\mu_r},
\end{equation*}
where  $E_{\mu_r}$ is the eigenprojector corresponding to $\mu_r\in \mathrm {Spec}_{A}(G)$, and
\begin{equation}
\label{ASpecDec2}
U_{A_G}(t)=\sum_{k\geq 0}\dfrac{(-\mathrm{i})^{k}A_G^{k}t^{k}}{k!}=\sum_{\mu_r\in \mathrm {Spec}_{A}(G)}\exp(-\mathrm{i}t\mu_{r})E_{\mu_r}.
\end{equation}
The \emph{adjacency eigenvalue support} of a vertex state $\mathbf{e}_{u}$ (or a vertex $u$) is defined as
$$
\mathrm{{supp}}_{A_G}(\mathbf{e}_{u})=\left \{ \mu\in \mathrm {Spec}_{A}(G)\mid E_{\mu}\mathbf{e}_{u}\neq \mathbf{0} \right \}.
$$
Two vertex states $\mathbf{e}_{u}$ and $\mathbf{e}_{v}$ (or two vertices $u$ and $v$) are called \emph{adjacency strongly cospectral} if $E_{\mu}\mathbf{e}_{u}=\pm E_{\mu}\mathbf{e}_{v}$ for any $\mu\in \mathrm {Spec}_{A}(G)$.

Define
$$
\Lambda^{+}_{u,v}=\left\{\mu\in \mathrm{{supp}}_{A_G}(\mathbf{e}_{u}) \mid E_{\mu}\mathbf{e}_{u}=+E_{\mu}\mathbf{e}_{v}\right\},
$$
and
$$
\Lambda^{-}_{u,v}=\left\{\mu\in \mathrm{{supp}}_{A_G}(\mathbf{e}_{u}) \mid E_{\mu}\mathbf{e}_{u}=-E_{\mu}\mathbf{e}_{v}\right\}.
$$
Then $\mathbf{e}_{u}$ and $\mathbf{e}_{v}$ are adjacency strongly cospectral if and only if $$
\mathrm{{supp}}_{A_G}(\mathbf{e}_{u})=\Lambda^{+}_{u,v}\cup \Lambda^{-}_{u,v}.
$$

In particular, we say that $\mathbf{e}_{u}$ is adjacency strongly cospectral with itself. In this case, $\Lambda^{-}_{u,u}=\emptyset$ and $\mathrm{{supp}}_{A_G}(\mathbf{e}_{u})=\Lambda^{+}_{u,u}$.

Let $\mathbb{Q}$ be the set of rational numbers. The following two results give necessary and sufficient conditions for a graph to have Pair-LPST and PST, respectively.

\begin{lemma}\emph{(See \cite[Theorem~3.9]{ChG19})}\label{Coutinho}
Let $\mathbf{e}_a-\mathbf{e}_b$ and $\mathbf{e}_c-\mathbf{e}_d$ be two distinct pair states in a graph $G$. Set
$S=\mathrm{supp}_{L_G}(\mathbf{e}_{a}-\mathbf{e}_{b})=\left \{ \theta_{r} \mid  0\le r \le k \right\}$ with $\theta_{0}\in \Lambda^{+}_{ab,cd}$. Then $G$ has Pair-LPST between $\mathbf{e}_a-\mathbf{e}_b$ and $\mathbf{e}_c-\mathbf{e}_d$ if and only if all of
the following hold:
\begin{itemize}
\item[\rm (a)]  $\mathbf{e}_a-\mathbf{e}_b$ and $\mathbf{e}_c-\mathbf{e}_d$ are Laplacian strongly cospectral;
 \item[\rm (b)] The eigenvalues in $S$ are either all integers or all quadratic integers. Moreover, there is a square-free integer $\Delta$ such that each $\theta \in S$ is
    a quadratic integer in $\mathbb{Q}(\sqrt{\Delta})$, and the difference of any two eigenvalues in $S$ is an integer multiple of $\sqrt{\Delta}$. Here, we allow $\Delta=1$ for the case where all eigenvalues in $S$ are integers;
\item[\rm (c)] Let $g=\gcd\left(\left\{\frac{\theta_{0}-\theta_{r}}{\sqrt{\Delta}}\right\}^{k}_{r=0}\right)$. Then
\begin{itemize}
\item[\rm(i)] $\theta_{r}\in\Lambda^{+}_{ab,cd}$ if and only if $\frac{\theta_{0}-\theta_{r}}{g\sqrt{\Delta}}$ is even, and
 \item[\rm(ii)] $\theta_{r}\in\Lambda^{-}_{ab,cd}$ if and only if $\frac{\theta_{0}-\theta_{r}}{g\sqrt{\Delta}}$ is odd.
\end{itemize}
\end{itemize}

 If these conditions hold and Pair-LPST occurs between $\mathbf{e}_a-\mathbf{e}_b$ and $\mathbf{e}_c-\mathbf{e}_d$ at time $\tau$, then the minimum time is
 $\tau_0=\frac{\pi}{g\sqrt{\Delta}}$.
\end{lemma}

\begin{lemma}\emph{ (See  \cite[Theorem~2.4.4] {Coutinho14})}\label{PST}
Assume that $u$ and $v$ are two distinct vertices of a graph $G$. Set $S=\mathrm{supp}_{A_G}(\mathbf{e}_{u})=\left \{ \mu_{r} \mid  0\le r \le k \right\}$ with $\mu_{0}\in \Lambda^{+}_{u,v}$. Then $G$ has PST between $u$ and $v$ if and only if all the following conditions hold:
\begin{itemize}
\item[\rm (a)] $u$ and $v$ are adjacency strongly cospectral;
 \item[\rm (b)] The non-zero eigenvalues of $S$ are either all integers or all quadratic integers. Moreover, if the non-zero eigenvalues of $S$ are all quadratic integers, then there exist integers $a,  b_{0}, \dots,  b_{k}$ and square-free integers $\Delta$ satisfying
 $$
 \mu_{r}=\frac{1}{2}(a+b_{r}\sqrt{\Delta}),~~ r=0,\dots,k.
 $$
\item[\rm (c)] Let $g=\gcd\left(\left\{\frac{\mu_{0}-\mu_{r}}{\sqrt{\Delta}}\right\}^{k}_{r=0}\right)$. Then
\begin{itemize}
\item[\rm(i)] $\mu_{r}\in\Lambda^{+}_{u,v}$ if and only if $\frac{\mu_{0}-\mu_{r}}{g\sqrt{\Delta}}$ is even;
 \item[\rm(ii)] $\mu_{r}\in\Lambda^{-}_{u,v}$ if and only if $\frac{\mu_{0}-\mu_{r}}{g\sqrt{\Delta}}$ is odd.
\end{itemize}
\end{itemize}

If the above conditions hold, then $G$ has PST between $u$ and $v$ at time $\tau \in \mathbb{R}$, where $\tau$ is an odd multiple of $\tau_0=\frac{\pi}{g\sqrt{\Delta}}$.
\end{lemma}

\begin{lemma}\label{Upair2}
Let  $\mathbf{e}_a-\mathbf{e}_b$ and $\mathbf{e}_c-\mathbf{e}_d$ be two Laplacian strongly cospectral pair states in a graph $G$. Set $\theta_{0}\in \Lambda^{+}_{ab,cd}$ and $\chi=\exp(-\mathrm{i}\tau_{0}\theta_{0})$. If $G$ has Pair-LPST between $\mathbf{e}_a-\mathbf{e}_b$ and $\mathbf{e}_c-\mathbf{e}_d$ at the minimum time $\tau_{0}$, then for any $k\in \mathbb{Z} $, we have
\begin{align*}
U_{L_{G}}(2k\tau_{0})(\mathbf{e}_{a}-\mathbf{e}_{b}) &=\chi^{2k} (\mathbf{e}_{a}-\mathbf{e}_{b}),\\
U_{L_{G}}((2k+1)\tau_{0})(\mathbf{e}_{a}-\mathbf{e}_{b}) &=\chi^{2k+1} (\mathbf{e}_{c}-\mathbf{e}_{d}).
\end{align*}
\end{lemma}

%\xg{If $G$ has Pair-LPST between $\mathbf{e}_a-\mathbf{e}_b$ and $\mathbf{e}_c-\mathbf{e}_d$ at time $\tau_{0}$, then $\mathbf{e}_a-\mathbf{e}_b$ and $\mathbf{e}_c-\mathbf{e}_d$ are Laplacian strongly cospectral. So, is the condition ``Let $\mathbf{e}_a-\mathbf{e}_b$ and $\mathbf{e}_c-\mathbf{e}_d$ be two Laplacian strongly cospectral pair states in a graph $G$.'' necessary first in Lemma \ref{Upair2}?}

\begin{proof}
Set $S=\mathrm{{supp}}_{L_G}(\mathbf{e}_{a}-\mathbf{e}_{b})$. Note that $\mathbf{e}_a-\mathbf{e}_b$ and $\mathbf{e}_c-\mathbf{e}_d$ are Laplacian strongly cospectral, then $S=\Lambda^{+}_{ab,cd}\cup \Lambda^{-}_{ab,cd}$. Assume that $G$ has Pair-LPST between $\mathbf{e}_a-\mathbf{e}_b$ and $\mathbf{e}_c-\mathbf{e}_d$ at the minimum time $\tau_{0}$. It follows from \eqref{LSpecDec2} and Lemma \ref{Coutinho} that
\begin{equation}
\begin{aligned}\label{Upair}
U_{L_{G}}(\tau_{0})(\mathbf{e}_{a}-\mathbf{e}_{b})
=&\sum_{\theta_{r}\in S}\exp(-\mathrm{i}\tau_{0}\theta_{r})F_{\theta_r}(\mathbf{e}_{a}-\mathbf{e}_{b})\\
=&\exp(-\mathrm{i}\tau_{0}\theta_{0})\sum_{\theta_{r}\in S}\exp(\mathrm{i}\tau_{0}(\theta_{0}-\theta_{r}))F_{\theta_r}(\mathbf{e}_{a}-\mathbf{e}_{b})\\
=&\exp(-\mathrm{i}\tau_{0}\theta_{0})\left(\sum_{\theta_{r}\in \Lambda_{ab,cd}^{+}}\exp(\mathrm{i}\tau_{0}(\theta_{0}-\theta_{r}))F_{\theta_r}(\mathbf{e}_{a}-\mathbf{e}_{b}) \right.\\&\left.
+\sum_{\theta_{r}\in \Lambda_{ab,cd}^{-}}\exp(\mathrm{i}\tau_{0}(\theta_{0}-\theta_{r}))F_{\theta_r}(\mathbf{e}_{a}-\mathbf{e}_{b})\right)\\
=&\exp(-\mathrm{i}\tau_{0}\theta_{0})\left(\sum_{\theta_{r}\in \Lambda_{ab,cd}^{+}}F_{\theta_r}(\mathbf{e}_{c}-\mathbf{e}_{d})+
\sum_{\theta_{r}\in \Lambda_{ab,cd}^{-}}F_{\theta_r}(\mathbf{e}_{c}-\mathbf{e}_{d})\right)\\
=&\exp(-\mathrm{i}\tau_{0}\theta_{0})(\mathbf{e}_{c}-\mathbf{e}_{d}).
\end{aligned}
\end{equation}

Similarly, the transition matrix $U_{L_{G}}(\tau_{0})$ also satisfies that
\begin{equation}\label{Upair1}
U_{L_{G}}(\tau_{0})(\mathbf{e}_{c}-\mathbf{e}_{d})=\exp(-\mathrm{i}\tau_{0}\theta_{0})(\mathbf{e}_{a}-\mathbf{e}_{b}).
\end{equation}

Together with (\ref{Upair}) and (\ref{Upair1}), it is easy to verify the conclusion by induction.
\qed
\end{proof}

Similar to Lemma \ref{Upair2}, by Lemma \ref{PST}, we get the following  result.
\begin{lemma}\label{Upst2}
Suppose that $u$ and $v$ are two adjacency strongly cospectral vertices in a graph $G$. Set $\mu_{0}\in \Lambda^{+}_{u,v}$ and $\lambda=\exp(-\mathrm{i}\tau_{0}\mu_{0})$. If $G$ has PST between $u$ and $v$ at the minimum time $\tau_{0}$, then for any $k\in \mathbb{Z} $, we have
\begin{equation*}
\begin{aligned}
U_{A_{G}}(2k\tau_{0})\mathbf{e}_{u}&=\lambda^{2k} \mathbf{e}_{u},\\
U_{A_{G}}((2k+1)\tau_{0})\mathbf{e}_{u}&=\lambda^{2k+1} \mathbf{e}_{v}.
\end{aligned}
\end{equation*}
\end{lemma}
% In fact, it is known \cite[Theorem~3.2.5]{QCh18} that if $\mathbf e_{a}-\mathbf e_{b}$ is %an edge state of $G$ with the eigenvalue support $S$ and $\theta \in S$, then all
%its algebraic conjugates are also in $S$. According to the proof of Lemma \ref{conjugates} in %\cite{QCh18}, we can easily verify the following result.

%\begin{lemma}\label{conjugates}
%Let $\mathbf e_{a}-\mathbf e_{b}$ be an pair state of $G$ with eigenvalue support $S$. If %$\theta \in S$, then all its algebraic conjugates are also in $S$.
%\end{lemma}

%By  Lemma \ref{conjugates}, $\frac{1}{2}(a-b_{\pm}\sqrt{\Delta})$, which are the %algebraic conjugates of $\frac{1}{2}(a+b_{\pm}\sqrt{\Delta})$, are in $S$. Lemma
%\ref{Coutinho}(b) tell us that the difference of any two eigenvalues of $S$ is an integer %multiple of $\sqrt{\Delta}$. Thus, $b_{\pm}$ are integers.

%%%%%%%%%%%%%%%%%%%%%

\section{Pair-LPST in tensor product of graphs} \label{Tensor product}

Let $G$ and $H$  be two graphs with adjacency matrices $A_G$ and $A_H$, respectively. The \emph{tensor product} of $G$ and $H$, denoted by $G\times H$, is the graph with the vertex set
$$
V_{G\times H}=\left \{(u,v) \mid u\in V_G, v\in V_H \right \},
$$
whose adjacency matrix is the Kronecker product $A_G\otimes A_H$. In this section, we investigate the existence of Pair-LPST in the tensor product of two regular graphs and construct several families of graphs having Pair-LPST.

\begin{lemma}\label{UMatrice}
Let $G$ be an $r_1$-regular graph and $H$ an $r_2$-regular graph. Set $L_G=\sum_{r=0}^{p}\theta_{r}F_{\theta_{r}}$. Then
 \begin{equation} \label{UGH}
U_{L_{G\times H}}(t)=e^{-\mathrm{i}r_{1}r_{2}t}\sum_{r=0}^{p}F_{\theta_{r}}\otimes U_{A_H}((\theta_{r}-r_1)t).
\end{equation}
\end{lemma}

\begin{proof}
Since $G$ is an $r_1$-regular graph, it follows from the spectral decomposition of $L_G$ that $A_G=\sum_{r=0}^{p}(r_1-\theta_{r})F_{\theta_{r}}$. Let $\mu_0>\mu_1>\cdots > \mu_{d}$ be all the distinct adjacency eigenvalues of $H$ and $E_{\mu_{s}}$ the eigenprojector corresponding to $\mu_{s}~(0\le s\le d)$. Then the spectral decomposition of $A_H$ is $A_H=\sum_{s=0}^{d}\mu_{s}E_{\mu_{s}}$. It is easy to verify that
\begin{equation} \label{UGH2}
 A_{G}\otimes A_{H}=\sum_{r=0}^{p}\sum_{s=0}^{d}(r_1-\theta_{r})\mu_{s}F_{\theta_{r}}\otimes E_{\mu_{s}}.
\end{equation}
Note that $G\times H$ is $r_1r_2$-regular with the Laplacian matrix
\begin{equation} \label{UGH1}
L_{G\times H}=r_1r_2I_{n\times m}-A_{G}\otimes A_{H}.
\end{equation}
Combining (\ref{UGH2}) with (\ref{UGH1}), we obtain
$$
L_{G\times H}=\sum_{r=0}^{p}\sum_{s=0}^{d}(r_{1}r_{2}-(r_1-\theta_{r})\mu_{s})F_{\theta_{r}}\otimes E_{\mu_{s}}.
$$
It follows that
\begin{equation} \label{UGH3}
U_{L_{G\times H}}(t)=e^{-\mathrm{i}r_{1}r_{2}t}\sum_{r=0}^{p}\sum_{s=0}^{d}F_{\theta_{r}}\otimes e^{\mathrm{i}t(r_1-\theta_{r})\mu_{s}}E_{\mu_{s}},
\end{equation}
which leads to (\ref{UGH}).
\qed
\end{proof}

\begin{theorem}
\label{tensor1}
Let $G$ be an $r_1$-regular graph with $n$ vertices and $H$ an $r_2$-regular graph with $m$ vertices. Assume that $\mathbf{e}^{n}_a-\mathbf{e}^{n}_b$ and $\mathbf{e}^{n}_c-\mathbf{e}^{n}_d$ are Laplacian strongly cospectral in $G$. Set
 $$
 S=\mathrm{supp}_{L_G}(\mathbf{e}^{n}_{a}-\mathbf{e}^{n}_{b})=\left \{\theta_{r} \mid  0\le r \le k \right\}, ~\theta_{0}\in\Lambda^{+}_{ab,cd}.
 $$
Suppose that $H$ admits PST between $w$ and $z$ at the minimum time $\tau$ with a phase factor $\lambda~(\left |\lambda  \right |=1 )$. Then $G\times H$ has Pair-LPST between $\mathbf e_{(a,w)}^{n\times m}-\mathbf e_{(b,w)}^{n\times m}$ and $\mathbf e_{(c,z)}^{n\times m}-\mathbf e_{(d,z)}^{n\times m}$ at time $t$ if and only if all the following conditions hold:
\begin{itemize}
\item[\rm (a)] $(\theta_{r}-r_1)t$ is an odd multiple of $\tau$ for each $\theta_{r}\in S$;
 \item[\rm (b)] $\lambda$ is a primitive p-th root of unity, where $p$ is even and $\lambda^{\frac{p}{2}}=-1$. And
\begin{itemize}
\item[\rm(i)] $\theta_{r}\in\Lambda^{+}_{ab,cd}$ if and only if $\frac{\theta_{0}t-\theta_{r}t}{\tau}\equiv 0 \pmod{p}$, and
 \item[\rm(ii)] $\theta_{r}\in\Lambda^{-}_{ab,cd}$ if and only if $\frac{\theta_{0}t-\theta_{r}t}{\tau}\equiv \frac{p}{2}\pmod{p}$.
\end{itemize}
\end{itemize}
\end{theorem}

\begin{proof}
Assume that $G\times H$ has Pair-LPST between $\mathbf e_{(a,w)}^{n\times m}-\mathbf e_{(b,w)}^{n\times m}$ and $\mathbf e_{(c,z)}^{n\times m}-\mathbf e_{(d,z)}^{n\times m}$ at time $t$ with a phase factor $\chi~(\left|\chi\right|=1)$ satisfying
\begin{align}\label{TS11}
U_{L_{G\times H}}(t)(\mathbf e_{a}^{n}-\mathbf e_{b}^{n})\otimes  \mathbf e_{w}^{m}=\chi(\mathbf e_{c}^{n}-\mathbf e_{d}^{n})\otimes  \mathbf e_{z}^{m}.
\end{align}
By the definition of Laplacian eigenvalue support, together with (\ref {UGH}) and (\ref {TS11}), we have
 \begin{align}\label{NPG}
e^{-\mathrm{i}tr_{1}r_{2}}\sum_{r=0}^{k}F_{\theta_{r}}(\mathbf e_{a}^{n}-\mathbf e_{b}^{n})\otimes U_{A_H}((\theta_{r}-r_1)t) \mathbf e_{w}^{m}
=\chi(\mathbf e_{c}^{n}-\mathbf e_{d}^{n})\otimes \mathbf e_{z}^{m}.
\end{align}

Multiplying both sides of (\ref{NPG}) on the left by $F_{\theta_{r}}\otimes I_m~(0\le r \le k)$, we get
\begin{align}\label{Upm}
 e^{-\mathrm{i}tr_{1}r_{2}}U_{A_H}((\theta_{r}-r_1)t)\mathbf e_{w}^{m}=\pm \chi \mathbf e_{z}^{m},
\end{align}
since $F_{\theta_{r}}(\mathbf e_{a}^{n}-\mathbf e_{b}^{n})=\pm F_{\theta_{r}}(\mathbf e_{c}^{n}-\mathbf e_{d}^{n})$. Moreover, (\ref{Upm}) implies that $H$ has PST between $z$ and $w$ at time $(\theta_{r}-r_1)t$. By Lemma \ref{PST}, (a) follows. Furthermore, we assume that $(\theta_{r}-r_1)t=k_r\tau$, where $k_r~(0 \le r\le k)$ is odd.

If $\theta_{r}\in\Lambda^{+}_{ab,cd}$, by Lemma \ref{Upst2} and (\ref{Upm}), we have
\begin{align*}
\lambda^{k_{r}}\mathbf e_{z}^{m}=U_{A_H}((\theta_{r}-r_1)t)\mathbf e_{w}^{m}
 =\chi e^{\mathrm{i}tr_{1}r_{2}}\mathbf e_{z}^{m}=U_{A_H}((\theta_{0}-r_1)t)\mathbf e_{w}^{m}=\lambda^{k_{0}}\mathbf e_{z}^{m},
\end{align*}
which implies that $\lambda$ is a primitive p-th root of unity and
$$
\frac{\theta_{0}t-\theta_{r}t}{\tau}\equiv 0 \pmod{p}.
$$

If $\theta_{r}\in\Lambda^{-}_{ab,cd}$, by Lemma \ref{Upst2} and (\ref{Upm}), we get
$$
 \lambda^{k_{r}}\mathbf e_{z}^{m}=U_{A_H}((\theta_{r}-r_1)t)\mathbf e_{w}^{m}
 =-\chi e^{\mathrm{i}tr_{1}r_{2}}\mathbf e_{z}^{m}=-U_{A_H}((\theta_{0}-r_1)t)\mathbf e_{w}^{m}=-\lambda^{k_{0}}\mathbf e_{z}^{m}.
$$
It follows that $\lambda^{k_{0}-k_{r}}=-1$, which means that $p$ is even and satisfies that $\lambda^{\frac{p}{2}}=-1$,
$$
\frac{\theta_{0}t-\theta_{r}t}{\tau}\equiv \frac{p}{2}\pmod{p}.
$$
This proves (b).

Conversely, suppose that (a) holds. By Lemma \ref{Upst2}, we get
$$
U_{A_H}((\theta_{0}-r_1)t) \mathbf e_{w}^{m}=\lambda^{k_0} \mathbf e_{z}^{m}.
$$
According to the definition of Laplacian eigenvalue support and (\ref{UGH}), we have
\begin{equation}
\begin{aligned}\label{Uten}
U_{L_{G\times H}}(t)(\mathbf e_{a}^{n}-\mathbf e_{b}^{n})\otimes  \mathbf e_{w}^{m}
&=e^{-\mathrm{i}tr_{1}r_{2}}\sum_{r=0}^{k}F_{\theta_{r}}(\mathbf e_{a}^{n}-\mathbf e_{b}^{n})\otimes U_{A_H}((\theta_{r}-r_1)t) \mathbf e_{w}^{m}\\
&=e^{-\mathrm{i}tr_{1}r_{2}}\sum_{r=0}^{k}F_{\theta_{r}}(\mathbf e_{a}^{n}-\mathbf e_{b}^{n})\otimes U_{A_H}(\theta_{r}t-\theta_{0}t)U_{A_H}(\theta_{0}t-r_1t) \mathbf e_{w}^{m}.
\end{aligned}
\end{equation}

If $\theta_{r}\in\Lambda^{+}_{ab,cd}$, then by (i) and Lemma \ref{Upst2}, we obtain
\begin{align}\label{Upair+1}
U_{A_H}(\theta_{r}t-\theta_{0}t)U_{A_H}(\theta_{0}t-r_1t) \mathbf e_{w}^{m}=\lambda^{k_0} \mathbf e_{z}^{m}.
\end{align}

If $\theta_{r}\in\Lambda^{-}_{ab,cd}$, then by (ii) and Lemma \ref{Upst2}, we get
\begin{align}\label{Upair-1}
U_{A_H}(\theta_{r}t-\theta_{0}t)U_{A_H}(\theta_{0}t-r_1t) \mathbf e_{w}^{m}=-\lambda^{k_0} \mathbf e_{z}^{m}.
\end{align}
Plugging (\ref{Upair+1}) and (\ref{Upair-1}) into (\ref{Uten}), we have
$$
U_{L_{G\times H}}(t)(\mathbf e_{a}^{n}-\mathbf e_{b}^{n})\otimes  \mathbf e_{w}^{m}=e^{-\mathrm{i}tr_{1}r_{2}}\lambda^{k_0}(\mathbf e_{c}^{n}-\mathbf e_{d}^{n})\otimes  \mathbf e_{z}^{m},
$$
which proves that $G\times H$ has Pair-LPST between $\mathbf e_{(a,w)}^{n\times m}-\mathbf e_{(b,w)}^{n\times m}$ and $\mathbf e_{(c,z)}^{n\times m}-\mathbf e_{(d,z)}^{n\times m}$.
\qed
\end{proof}

Recall that $\mathbf e^n_{a}-\mathbf e^n_{b}$ is Laplacian strongly cospectral with itself. If we take $\mathbf{e}^{n}_a-\mathbf{e}^{n}_b=\mathbf{e}^{n}_c-\mathbf{e}^{n}_d$ in Theorem \ref{tensor1}, then we have the following result immediately.

\begin{cor}
\label{cor1}
Let $G$ be an $r_1$-regular graph with $n$ vertices and $H$ an $r_2$-regular graph with $m$ vertices. For a pair state $\mathbf{e}^{n}_a-\mathbf{e}^{n}_b$ in $G$, let
$$
S=\mathrm{supp}_{L_G}(\mathbf{e}^{n}_{a}-\mathbf{e}^{n}_{b})=\left \{\theta_{r} \mid  0\le r \le k \right\}.
$$
Suppose that $H$ admits PST between $w$ and $z$ at the minimum time $\tau$ with phase factor $\lambda~(\left |\lambda  \right |=1 )$. Then $G\times H$ has Pair-LPST between $\mathbf e_{(a,w)}^{n\times m}-\mathbf e_{(b,w)}^{n\times m}$ and $\mathbf e_{(a,z)}^{n\times m}-\mathbf e_{(b,z)}^{n\times m}$ at time $t$ if and only if all the following conditions hold:
\begin{itemize}
  \item[\rm (a)] $(\theta_{r}-r_1)t$ is an odd multiple of $\tau$ for each $\theta_{r}\in S$;
  \item[\rm (b)] $\lambda$ is a primitive p-th root of unity, and for any $\theta_{r}, \theta_{s}\in S$,
 $$\frac{\theta_{r}t-\theta_{s}t}{\tau}\equiv 0 \pmod{p}.$$
\end{itemize}
\end{cor}

%\begin{proof}
%Recall that $\mathbf e_{a}-\mathbf e_{b}$ is Laplacian strongly cospectral with itself, and %$S=\Lambda^{+}_{ab,ab}$.
%Here we let $\mathbf{e}^{n}_a-\mathbf{e}^{n}_b=\mathbf{e}^{n}_c-\mathbf{e}^{n}_d$ %and it is easy to conclude (a) from Theorem \ref{tensor1}. Note that %$S=\Lambda^{+}_{ab,ab}$. For each $\theta_{r},~\theta_{s}\in S$, it follows from %Theorem \ref{tensor1} (b) (i) that
%$$\frac{\theta_{r}t-\theta_{s}t}{\tau}\equiv 0 \pmod{p}.$$
%The proof is complete.
%\qed
%\end{proof}

Theorem \ref{tensor1} and Corollary \ref{cor1}  imply that we can construct tensor products having Pair-LPST using a graph admitting PST, as shown in the following example.

\begin{example}\label{Example1}
{\em Let $K_{n}$ be a complete graph with $n\ge3$ vertices and let $P_2$ denote the path with the vertex set $\left\{0,1\right\}$. It is known that $L_{K_n}$ has the spectral decomposition $L_{K_n}=nF_{n}+0F_{0}$, where
\begin{equation}
\begin{aligned}\label{UDKn}
F_{0}=\frac{1}{n}J_{n},~~F_{n}=I-\frac{1}{n}J_{n}.
\end{aligned}
\end{equation}
For any vertices $a, b\in V(K_n)$, it is easy to verify that
$$
\mathrm{{supp}}_{L_{K_n}}\left(\mathbf e_{a}^{n}-\mathbf e_{b}^{n} \right )=\left\{n\right\}.
$$
Note that $P_2$ has PST between $\mathbf{e}^{2}_{1}$ and $\mathbf{e}^{2}_{0}$ at time $\tau=\frac{\pi}{2}$ with a phase factor $\lambda=-\mathrm{i}$. By Corollary \ref{cor1}, we conclude that $K_{n}\times P_2$ has Pair-LPST between $\mathbf e_{(a,1)}^{n\times 2}-\mathbf e_{(b,1)}^{n\times 2}$ and $\mathbf e_{(a,0)}^{n\times 2}-\mathbf e_{(b,0)}^{n\times 2}$ at time $t=\frac{\pi}{2}$. }
\end{example}

In Theorem \ref{tensor1}, we give a necessary and sufficient condition for the tensor product of two regular graphs $G$ and $H$ to have Pair-LPST when $H$ admits PST. In the following, we give a necessary and sufficient condition for the tensor product of two regular graphs $G$ and $H$ to have Pair-LPST when $H$ admits Pair-PST with respect to $A_H$. Before proceeding, we give the following results.

Let $G$ be an $r$-regular graph with $n$ vertices. If $G$ has Pair-LPST between $\mathbf{e}^{n}_{a}-\mathbf{e}^{n}_{b}$ and $\mathbf{e}^{n}_{c}-\mathbf{e}^{n}_{d}$ with a phase factor $\chi$, then we have
\begin{align}\label{AequalL}
U_{L_G}(\tau)(\mathbf{e}^{n}_{a}-\mathbf{e}^{n}_{b})
=\exp(-\mathrm{i}\tau r)U_{A_G}(-\tau)(\mathbf{e}^{n}_{a}-\mathbf{e}^{n}_{b})
=\chi(\mathbf{e}^{n}_{c}-\mathbf{e}^{n}_{d}),
\end{align}
which implies that the existence of Pair-LPST is equivalent to that of Pair-PST with respect to $A_G$.

Similar to Lemma \ref{Upair2},  if $G$ has Pair-PST with respect to $A_G$ between $\mathbf{e}^{n}_{a}-\mathbf{e}^{n}_{b}$ and $\mathbf{e}^{n}_{c}-\mathbf{e}^{n}_{d}$ at the minimum time $\tau$ with a phase factor $\lambda~(\left |\lambda  \right |=1 )$, then we have
\begin{align}
\label{HTS1}  U_{A_G}(2k\tau)(\mathbf{e}^{n}_{a}-\mathbf{e}^{n}_{b}) &=\lambda^{2k}(\mathbf{e}^{n}_{a}-\mathbf{e}^{n}_{b}),\\
 \label{HTS}  U_{A_G}((2k+1)\tau)(\mathbf{e}^{n}_{a}-\mathbf{e}^{n}_{b}) &=\lambda^{2k+1}(\mathbf{e}^{n}_{c}-\mathbf{e}^{n}_{d}).
\end{align}

\begin{theorem}\label{tensor2}
Let $G$ be an $r_1$-regular graph with $n$ vertices and $H$ an $r_2$-regular graph with $m$ vertices. Let $w$ and $v$ be two adjacency strongly cospectral vertices of $G$. Set
$$
S=\mathrm{supp}_{A_G}(\mathbf{e}^{n}_{w})=\left \{\mu_{r} \mid  0\le r \le k \right\}, ~ \mu_{0}\in\Lambda^{+}_{w,z}.
$$
Suppose that $H$ admits Pair-PST with respect to $A_H$ between $\mathbf{e}^{m}_{a}-\mathbf{e}^{m}_{b}$ and $\mathbf{e}^{m}_{c}-\mathbf{e}^{m}_{d}$ at the minimum time $\tau$ with a phase factor $\lambda~(\left |\lambda  \right |=1 )$. Then $G\times H$ has Pair-LPST between $\mathbf e_{(w,a)}^{n\times m}-\mathbf e_{(w,b)}^{n\times m}$ and $\mathbf e_{(z, c)}^{n\times m}-\mathbf e_{(z,d)}^{n\times m}$ at time $t$ if and only if all the following conditions hold:
\begin{itemize}
  \item[\rm (a)] $\mu_{r}t$ is an odd multiple of $\tau$ for each $\mu_{r}\in S$;
  \item[\rm (b)] $\lambda$ is a primitive p-th root of unity, where $p$ is even and $\lambda^{\frac{p}{2}}=-1$. And
\begin{itemize}
  \item[\rm(i)] $\mu_{r}\in\Lambda^{+}_{w,z}$ if and only if $\frac{\mu_{0}t-\mu_{r}t}{\tau}\equiv 0 \pmod{p}$; and
  \item[\rm(ii)] $\mu_{r}\in\Lambda^{-}_{w,z}$ if and only if $\frac{\mu_{0}t-\mu_{r}t}{\tau}\equiv \frac{p}{2} \pmod{p}$.
\end{itemize}
\end{itemize}
\end{theorem}

\begin{proof}
Assume that the spectral decomposition of $A_G$ is $A_G=\sum\limits_{\mu_{r}\in \mathrm{Spec}_A(G)}\mu_{r}E_{\mu_{r}}$. By (\ref{UGH}), we have
\begin{equation} \label{TS2Euqil}
U_{L_{G\times H}}(t)=e^{-\mathrm{i}r_{1}r_{2}t}\sum_{\mu_{r}\in \mathrm{Spec}_A(G)}E_{\mu_r}\otimes U_{A_H}(-\mu_r t).
\end{equation}
If $G\times H$ has Pair-LPST between $\mathbf e_{(w,a)}^{n\times m}-\mathbf e_{(w,b)}^{n\times m}$ and $\mathbf e_{(z, c)}^{n\times m}-\mathbf e_{(z,d)}^{n\times m}$, then there exists a time $t$ and a phase factor $\chi~(\left|\chi\right|=1)$  such that
\begin{align*}
U_{L_{G\times H}}(t) \mathbf e_{w}^{n}\otimes (\mathbf e_{a}^{m}-\mathbf e_{b}^{m})=\chi\mathbf e_{z}^{n}\otimes (\mathbf e_{c}^{m}-\mathbf e_{d}^{m}).
\end{align*}
By the definition of adjacency eigenvalue support and (\ref{TS2Euqil}), we obtain
\begin{align}\label{NPGc}
e^{-\mathrm{i}tr_{1}r_{2}}\sum_{r=0}^{k}E_{\mu_r} \mathbf e_{w}^{n}\otimes U_{A_H}(-\mu_rt) (\mathbf e_{a}^{m}-\mathbf e_{b}^{m})
=\chi\mathbf e_{z}^{n}\otimes (\mathbf e_{c}^{m}-\mathbf e_{d}^{m}).
\end{align}

Left-multiplying $E_{\mu_r}\otimes I_m~(0\le r \le k)$ on both sides of (\ref{NPGc}), we get
\begin{align}\label{NPGc1}
\pm e^{-\mathrm{i}tr_{1}r_{2}}U_{A_H}(-\mu_{r}t)(\mathbf e_{a}^{m}-\mathbf e_{b}^{m})=\chi (\mathbf e_{c}^{m}-\mathbf e_{d}^{m}),
\end{align}
since $E_{\mu_r}\mathbf e_{w}^{n}=\pm E_{\mu_r}\mathbf e_{z}^{n}$. Thus, for any $\mu_{r}\in S$, $H$ has Pair-PST with respect to $A_H$ between $\mathbf e_{a}^{m}-\mathbf e_{b}^{m}$ and $\mathbf e_{c}^{m}-\mathbf e_{d}^{m}$ at time $-\mu_{r}t$. Since $H$ is regular, Pair-PST in $H$ with respect to $A_H$ is equivalent to Pair-LPST. By Lemma \ref{Coutinho}, we get (a).

The proof of (b) is similar to that of Theorem \ref{tensor1} (b), hence we omit the details.

Conversely, set $-\mu_{0}t=k_0\tau$, where $k_0$ is odd. By (\ref{HTS}), we have
$$
U_{A_H}(-\mu_{0}t) (\mathbf e_{a}^{m}-\mathbf e_{b}^{m})=\lambda^{k_0}(\mathbf e_{c}^{m}-\mathbf e_{d}^{m}).
$$
By the definition of adjacency eigenvalue support and (\ref{TS2Euqil}), we obtain
\begin{equation}
\begin{aligned}\label{Utenc}
U_{L_{G\times H}}(t) \mathbf e_{w}^{n}\otimes (\mathbf e_{a}^{m}-\mathbf e_{b}^{m})
=e^{-\mathrm{i}tr_{1}r_{2}}\sum_{r=0}^{k}E_{\mu_r}\mathbf e_{w}^{n}\otimes U_{A_H}(\mu_{0}t-\mu_{r}t)U_{A_H}(-\mu_{0}t) (\mathbf e_{a}^{m}-\mathbf e_{b}^{m}).
\end{aligned}
\end{equation}
For $\mu_{r}\in\Lambda^{+}_{w,z}$, by Theorem \ref{tensor2} (i), (\ref{HTS1}) and (\ref{HTS}), we get
$$
U_{A_H}(\mu_{0}t-\mu_{r}t)U_{A_H}(-\mu_{0}t) (\mathbf e_{a}^{m}-\mathbf e_{b}^{m})=\lambda^{k_0} (\mathbf e_{c}^{m}-\mathbf e_{d}^{m}).
$$
For $\mu_{r}\in\Lambda^{-}_{w,z}$, by Theorem \ref{tensor2} (ii), (\ref{HTS1}) and (\ref{HTS}), we get
$$
U_{A_H}(\mu_{0}t-\mu_{r}t)U_{A_H}(-\mu_{0}t) (\mathbf e_{a}^{m}-\mathbf e_{b}^{m})=-\lambda^{k_0} (\mathbf e_{c}^{m}-\mathbf e_{d}^{m}).
$$
Then by (\ref{Utenc}), we obtain
$$
U_{L_{G\times H}}(t)  \mathbf e_{w}^{n}\otimes (\mathbf e_{a}^{m}-\mathbf e_{b}^{m})
=e^{-\mathrm{i}tr_{1}r_{2}}\lambda^{k_0} \mathbf e_{z}^{n}\otimes   (\mathbf e_{c}^{m}-\mathbf e_{d}^{m}).
$$
Therefore, $G\times H$ has Pair-LPST between $\mathbf e_{(w,a)}^{n\times m}-\mathbf e_{(w,b)}^{n\times m}$ and $\mathbf e_{(z, c)}^{n\times m}-\mathbf e_{(z,d)}^{n\times m}$.
\qed
\end{proof}

Recall that $\mathbf{e}^n_{w}$ is adjacency strongly cospectral with itself. If we take $\mathbf{e}^{n}_w=\mathbf{e}^{n}_z$ in Theorem \ref{tensor2}, then we immediately get the following result.

\begin{cor}\label{cor2}
Let $G$ be an $r_1$-regular graph with $n$ vertices and $H$ an $r_2$-regular graph with $m$ vertices. For a vertex $w$ in $G$, let
$$
S=\mathrm{supp}_{A_G}(\mathbf{e}^{n}_{w} )=\left \{\mu_{r} \mid  0\le r \le k \right\}.
$$
Suppose that $H$ admits Pair-LPST between $\mathbf{e}^{m}_{a}-\mathbf{e}^{m}_{b}$ and $\mathbf{e}^{m}_{c}-\mathbf{e}^{m}_{d}$ at the minimum time $\tau$ with phase factor $\lambda~(\left |\lambda  \right |=1 )$. Then $G\times H$ has Pair-LPST between $\mathbf e_{(w,a)}^{n\times m}-\mathbf e_{(w,b)}^{n\times m}$ and $\mathbf e_{(w, c)}^{n\times m}-\mathbf e_{(w,d)}^{n\times m}$ at time $t$ if and only if all the following conditions hold:
\begin{itemize}
  \item[\rm (a)] $\mu_{r}t$ is an odd multiple of $\tau$ for each $\mu_{r}\in S$;
  \item[\rm (b)] $\lambda$ is a primitive p-th root of unity, and for any $\mu_{r}, \mu_{s}\in S$,
  $$\frac{\mu_{r}t-\mu_{s}t}{\tau}\equiv 0 \pmod{p}.$$
\end{itemize}
\end{cor}

\begin{example}\label{Example2}
{\em Let $K_{2n}$ be a complete graph with $2n~(n\ge1)$ vertices and let $C_{4}$ denote the cycle with vertex set $\left\{0,1,2,3\right\}$.  Note that $A_{K_{2n}}=-E_{-1}+(2n-1)E_{2n-1}$, where
 \begin{equation}
\begin{aligned}\label{AUTK2n}
E_{-1}=I-\frac{1}{2n}J_{2n},~~E_{2n-1}=\frac{1}{2n}J_{2n}.
\end{aligned}
\end{equation}
It is known that $C_{4}$ has Pair-LPST between $\mathbf{e}^{4}_{0}-\mathbf{e}^{4}_{1}$ and $\mathbf{e}^{4}_{2}-\mathbf{e}^{4}_{3}$ at time $\tau=\frac{\pi}{2}$ with a phase factor 1, that is,
$$
U_{L_{C_{4}}}\left(\frac{\pi}{2}\right)\left(\mathbf{e}^{4}_{0}-\mathbf{e}^{4}_{1} \right)=\mathbf{e}^{4}_{2}-\mathbf{e}^{4}_{3}.
$$
Since $C_{4}$ is a 2-regular graph, by (\ref{AequalL}), we get
$$
U_{A_{C_{4}}}\left(\frac{\pi}{2}\right)\left(\mathbf{e}^{4}_{0}-\mathbf{e}^{4}_{1} \right)= -\mathbf{e}^{4}_{2}+\mathbf{e}^{4}_{3}.
$$
For any vertex $w\in V(K_{2n})$, it is easy to verify that $\mathrm{supp}_{A_{K_{2n}}}(\mathbf{e}^{2n}_{w})=\left \{ -1, 2n-1 \right\}$. Let $t=\frac{\pi}{2}$ and $\lambda=-1$. Then for each $\mu_{r}\in S$, $\mu_{r}t$ is an odd multiple of $\tau$ and $\lambda$ is a 2nd primitive root of unity. It follows that
 $$
 \frac{(2n-1)t-(-t)}{\tau}\equiv 0 \pmod{2}.
 $$
By Corollary \ref{cor2}, $K_{2n}\times C_4$ has Pair-LPST between $\mathbf e_{(w,0)}^{2n\times 4}-\mathbf e_{(w,1)}^{2n\times 4}$ and $\mathbf e_{(w,2)}^{2n\times 4}-\mathbf e_{(w,3)}^{2n\times 4}$ at time $t=\frac{\pi}{2}$.}
\end{example}

At the end of this section, we consider the existence of Pair-LPST for a more general form of the pair state $\mathbf e_{(a,w)}^{n\times m}-\mathbf e_{(b,z)}^{n\times m}$.

\begin{theorem}\label{tensor3}
Let $G$ be an $r_1$-regular graph with $n$ vertices and $H$ an $r_2$-regular graph with $m$ vertices. Let $a, b\in V(G)$ and $w, z\in V(H)$ and set
$$
S=\mathrm{supp}_{A_G}(\mathbf e_{a}^{n})\cup \mathrm{supp}_{A_G}(\mathbf e_{b}^{n})=\left \{\mu_{r} \mid  0\le r \le k \right\}.
$$
Suppose that $U_{A_H}(\tau)\mathbf e_{z}^{m}=\lambda \mathbf e_{w}^{m}$ holds at the minimum time $\tau$, where $\lambda$ is a primitive p-th root of unity. If there exists a number $t$ satisfying the following conditions:
\begin{itemize}
  \item[\rm (a)] for each $\mu_{r}\in S$, $\mu_{r}t$ is an odd multiple of $\tau$;
  \item[\rm (b)] for any $\mu_{r}, \mu_{s}\in S$,
  $$\frac{\mu_{r}t-\mu_{s}t}{\tau}\equiv 0 \pmod{p},$$
\end{itemize}
then $G\times H$ has Pair-LPST between $\mathbf e_{(a,w)}^{n\times m}-\mathbf e_{(b,z)}^{n\times m}$ and $\mathbf e_{(b,w)}^{n\times m}-\mathbf e_{(a,z)}^{n\times m}$ at time $t$.
\end{theorem}
\begin{proof}
Assume that $-\mu_{0}t=k_0 \tau$, where $k_{0}$ is odd. By Lemma \ref{Upst2}, we obtain
$$
U_{A_H}(-\mu_{0}t)\mathbf e_{z}^{m}=\lambda^{k_0}\mathbf e_{w}^{m}.
$$
According to (b), for any $\mu_{r}\in S$,
\begin{align}\label{TSORU}
U_{A_H}(-\mu_{r}t)\mathbf e_{z}^{m}=U_{A_H}(-\mu_{r}t+\mu_{0}t)U_{A_H}(-\mu_{0}t)\mathbf e_{z}^{m}=\lambda^{k_0}\mathbf e_{w}^{m}.
\end{align}
Since the sum of all adjacency eigenprojectors of $G$ is equal to the identity matrix, by the definition of $S$ and (\ref{TSORU}), we get
\begin{align}\label{TSORU1}
\sum_{r=0}^{k}(\mathbf e_{a}^{n})^{\top}E_{\mu_r}\mathbf e_{a}^{n}\otimes (\mathbf e_{w}^{m})^{\top}U_{A_H}(-\mu_{r}t)\mathbf e_{z}^{m}
=\sum_{r=0}^{k}(\mathbf e_{b}^{n})^{\top}E_{\mu_r}\mathbf e_{b}^{n}\otimes (\mathbf e_{z}^{m})^{\top}U_{A_H}(-\mu_{r}t)\mathbf e_{w}^{m}
=\lambda^{k_0}.
\end{align}
Combining (\ref{TS2Euqil}) and (\ref{TSORU1}), we have
\begin{align}\label{TSORU2}
(\mathbf e_{a}^{n}\otimes \mathbf e_{w}^{ m})^{\top}U_{L_{G\times H}}(t)\mathbf e_{a}^{n}\otimes \mathbf e_{z}^{ m}
=(\mathbf e_{b}^{n}\otimes \mathbf e_{z}^{ m})^{\top}U_{L_{G\times H}}(t)\mathbf e_{b}^{n}\otimes \mathbf e_{w}^{ m}=e^{-\mathrm{i}r_{1}r_{2}t}\lambda^{k_0}.
\end{align}
Notice that the transition matrix $U_{L_{G\times H}}(t)$ is unitary. By (\ref{TSORU2}), we get
\begin{align}\label{TSORU3}
(\mathbf e_{a}^{n}\otimes \mathbf e_{w}^{ m})^{\top}U_{L_{G\times H}}(t)\mathbf e_{b}^{n}\otimes \mathbf e_{w}^{ m}=(\mathbf e_{b}^{n}\otimes \mathbf e_{z}^{ m})^{\top}U_{L_{G\times H}}(t)\mathbf e_{a}^{n}\otimes \mathbf e_{z}^{ m}=0.
\end{align}
By (\ref{TSORU2}) and (\ref{TSORU3}), we have
\begin{align*}
&\left |\frac{1}{2}(\mathbf e_{(a,w)}^{n\times m}-\mathbf e_{(b,z)}^{n\times m})^{\top}U_{L_{G\times H}}(t)
(\mathbf e_{(b,w)}^{n\times m}-\mathbf e_{(a,z)}^{n\times m}) \right |^2\\
=&\frac{1}{4}\left |(\mathbf e_{a}^{n}\otimes \mathbf e_{w}^{ m})^{\top}U_{L_{G\times H}}(t)\mathbf e_{b}^{n}\otimes \mathbf e_{w}^{ m}
-(\mathbf e_{a}^{n}\otimes \mathbf e_{w}^{ m})^{\top}U_{L_{G\times H}}(t)\mathbf e_{a}^{n}\otimes \mathbf e_{z}^{ m}  \right.\\&\left.
-(\mathbf e_{b}^{n}\otimes \mathbf e_{z}^{ m})^{\top}U_{L_{G\times H}}(t)\mathbf e_{b}^{n}\otimes \mathbf e_{w}^{ m}
+(\mathbf e_{b}^{n}\otimes \mathbf e_{z}^{ m})^{\top}U_{L_{G\times H}}(t)\mathbf e_{a}^{n}\otimes \mathbf e_{z}^{ m}\right |^2 \\
=&\left |-e^{-\mathrm{i}r_{1}r_{2}t}\lambda^{k_0} \right |^2=1,
\end{align*}
which implies that $G\times H$ has Pair-LPST between $\mathbf e_{(a,w)}^{n\times m}-\mathbf e_{(b,z)}^{n\times m}$ and $\mathbf e_{(b,w)}^{n\times m}-\mathbf e_{(a,z)}^{n\times m}$.
\qed
\end{proof}

\begin{example}\label{Example3}
{\em Let $K_{4n}$ be a complete graph with $4n~(n\ge 1)$ vertices and let $P_{2}$ denote the path with the vertex set $\left\{0,1 \right\}$. For any two vertices $ a,b \in V(K_{4n})$, by (\ref{AUTK2n}), we have
$$
S=\mathrm{supp}_{A_{K_{4n}}}(\mathbf e_{a}^{4n})\cup \mathrm{supp}_{A_{K_{4n}}}(\mathbf e_{b}^{4n})=\left \{-1,4n-1 \right\}.
$$
Note that $P_2$ has PST between $\mathbf{e}^{2}_{0}$ and $\mathbf{e}^{2}_{1}$, and
$$
U_{A_{P_2}}\left(\frac{\pi}{2}\right)\mathbf{e}^{2}_{0}=-\mathrm{i}\mathbf{e}^{2}_{1}.
$$
Thus, $\tau=\frac{\pi}{2}$ and $p=4$. By Theorem \ref{tensor3}, we conclude that $K_{4n}\times P_2$ has Pair-LPST between $\mathbf e_{(a,0)}^{4n\times 2}-\mathbf e_{(b,1)}^{4n\times 2}$ and $\mathbf e_{(b,1)}^{4n\times 2}-\mathbf e_{(a,0)}^{4n\times 2}$ at time $\frac{\pi}{2}$ .}
\end{example}

\section{Pair-LPST in double cover of graphs}\label{Double cover}
Let $G$ and $H$ be two graphs that have the same vertex set $V$. The \emph{double cover} of $G$ and $H$, denoted by $G\ltimes H$, is the graph with the vertex set $\left\{0, 1\right\}\times V$, whose adjacency matrix is
$$
A_{G\ltimes H}=\begin{pmatrix}
 A_{G} & A_{H}\\
 A_{H} & A_{G}
\end{pmatrix}.
$$
According to \cite[Theorem 5.2]{Cohgo16}, the adjacency transition matrix of $G\ltimes H$ is
\begin{equation}\label{pairdouble-equation1}
U_{A_{G\ltimes H}}(t)=\frac{1}{2}
\begin{pmatrix}
 U_{A_{G}+A_{H}}(t)+U_{A_{G}-A_{H}}(t) & U_{A_{G}+A_{H}}(t)-U_{A_{G}-A_{H}}(t)\\
 U_{A_{G}+A_{H}}(t)-U_{A_{G}-A_{H}}(t) & U_{A_{G}+A_{H}}(t)+U_{A_{G}-A_{H}}(t)
\end{pmatrix}.
\end{equation}
If $G$ and $H$ are $r_{1}$-regular and $r_{2}$-regular graphs, respectively, then $G\ltimes H$ is a $(r_{1}+r_{2})$-regular graph with the Laplacian matrix
$$
L_{G\ltimes H}=(r_{1}+r_{2})I_{2n}-A_{G\ltimes H}.
$$
It follows that
  \begin{equation}\label{DCU}
U_{L_{G\ltimes H}}(t)=U_{(r_{1}+r_{2})I_{n}-A_{G\ltimes H}}(t)
=\exp(-\mathrm{i}t(r_{1}+r_{2}))U_{A_{G\ltimes H}}(-t).\\
  \end{equation}

In this section, we study the existence of Pair-LPST in the double cover of two regular graphs. Before proceeding, we give the following result.

\begin{lemma}\label{Xlem}
Let $X$ be a Hermitian matrix over $\mathbb{C}$. For any two pair states $\mathbf e_{a}-\mathbf e_{b}$ and $\mathbf e_{c}-\mathbf e_{d}$, the transition matrix of $X$ at any time $\tau$ satisfies that
$$
\left |(\mathbf e_{a}-\mathbf e_{b})^{\top}U_{X}(\tau)(\mathbf e_{c}-\mathbf e_{d})\right|\le2.
$$
\end{lemma}
\begin{proof}
Note that the transition matrix $U_{X}(\tau)$ is unitary. Then
\begin{align*}
\left \|U_{X}(\tau)(\mathbf e_{c}-\mathbf e_{d}) \right \|
&=\sqrt{\left( U_{X}(\tau)(\mathbf e_{c}-\mathbf e_{d}) \right)^{H} \cdot  U_{X}(\tau)(\mathbf e_{c}-\mathbf e_{d}) }\\
&=\sqrt{ (\mathbf e_{c}-\mathbf e_{d})^{H} U_{X}(\tau) ^{H}\cdot  U_{X}(\tau)(\mathbf e_{c}-\mathbf e_{d}) }\\
&=\sqrt{2},
\end{align*}
where $\ast^{H}$ denotes the conjugate transpose of $\ast$.
Since $\mathbf e_{a}-\mathbf e_{b}$ is a real vector, we have
\begin{align*}
\left |(\mathbf e_{a}-\mathbf e_{b})^{\top}U_{X}(\tau)(\mathbf e_{c}-\mathbf e_{d})\right|
&=\left |(\mathbf e_{a}-\mathbf e_{b})^{H}U_{X}(\tau)(\mathbf e_{c}-\mathbf e_{d})\right|\\
&=\left | \left \langle \mathbf e_{a}-\mathbf e_{b}, U_{X}(\tau)(\mathbf e_{c}-\mathbf e_{d})  \right \rangle \right|\\
&\le \left \| \mathbf e_{a}-\mathbf e_{b}  \right \| \cdot \left \| U_{X}(\tau)(\mathbf e_{c}-\mathbf e_{d}) \right \| =2.
\end{align*}
This completes the proof.
\qed\end{proof}

In the following, we use $\mathbf e_{i}\otimes \mathbf e_{u}^{n}~(i=0, 1)$ to denote the vertex state of the vertex $(i, u)$ in $G\ltimes H$, where $\mathbf e_{0}=(1, 0)^{\top}$, $\mathbf e_{1}=(0, 1)^{\top}$.

\begin{theorem}\label{DC1}
Let $G$ be an $r_1$-regular graph and $H$ an $r_2$-regular graph, which have the same vertex set $V$. Denote the adjacency matrices of $G$ and $H$ by $A_G$ and $A_H$, respectively.
\begin{itemize}
\item[\rm (a)] If $a, b\in V$, then $G\ltimes H$ has Pair-LPST between $\mathbf e_{0}\otimes(\mathbf e_{a}^{n}-\mathbf e_{b}^{n})$ and $\mathbf e_{1}\otimes(\mathbf e_{a}^{n}-\mathbf e_{b}^{n})$ if and only if there exists a time $\tau$ such that $\mathbf e_{a}^{n}-\mathbf e_{b}^{n}$ is periodic with respect to $A_{G}+A_{H}$ and $A_{G}-A_{H}$ with phase factors $\chi$ and $-\chi$, respectively.
\item[\rm (b)] If $a, b, c, d\in V$, then $G\ltimes H$ has Pair-LPST between $\mathbf e_{i}\otimes(\mathbf e_{a}^{n}-\mathbf e_{b}^{n})$ and $\mathbf e_{i}\otimes(\mathbf e_{c}^{n}-\mathbf e_{d}^{n})$~$(i=0, 1)$ if and only if there exists a time $\tau$ such that $A_{G}+A_{H}$ and $A_{G}-A_{H}$ have Pair-PST between $\mathbf e_{a}^{n}-\mathbf e_{b}^{n}$ and $\mathbf e_{c}^{n}-\mathbf e_{d}^{n}$ with the same phase factor $\chi$.
\item[\rm (c)] If $a, b, c, d\in V$, then $G\ltimes H$ has Pair-LPST between $\mathbf e_{0}\otimes(\mathbf e_{a}^{n}-\mathbf e_{b}^{n})$ and $\mathbf e_{1}\otimes(\mathbf e_{c}^{n}-\mathbf e_{d}^{n})$ if and only if there exists a time $\tau$ such that $A_{G}+A_{H}$ and $A_{G}-A_{H}$ have Pair-PST between $\mathbf e_{a}^{n}-\mathbf e_{b}^{n}$ and $\mathbf e_{c}^{n}-\mathbf e_{d}^{n}$ with phase factors $\chi$ and $-\chi$, respectively.
\end{itemize}
\end{theorem}

\begin{proof}
Assume that $G\ltimes H$ has Pair-LPST between $\mathbf e_{0}\otimes(\mathbf e_{a}^{n}-\mathbf e_{b}^{n})$ and $\mathbf e_{1}\otimes(\mathbf e_{a}^{n}-\mathbf e_{b}^{n})$ at time $\tau$, that is,
$$
\left |\frac{1}{2}(\mathbf e_{0}\otimes(\mathbf e_{a}^{n}-\mathbf e_{b}^{n}))^{\top}U_{L_{G\ltimes H}}(\tau)(\mathbf e_{1}\otimes(\mathbf e_{a}^{n}-\mathbf e_{b}^{n}))\right|^2=1.
$$
By (\ref{pairdouble-equation1}) and (\ref{DCU}), the above equation leads to
\begin{align}
\label{UCA}
\left | (\mathbf e_{a}^{n}-\mathbf e_{b}^{n})^{\top}(U_{A_{G}+A_{H}}(-\tau)-U_{A_{G}-A_{H}}(-\tau))(\mathbf e_{a}^{n}-\mathbf e_{b}^{n}) \right|^2=16.
\end{align}
By Lemma \ref{Xlem}, Equation (\ref{UCA}) holds if and only if there exists a phase factor $\chi~(\left |\chi \right |=1 )$ such that
$$
U_{A_{G}+A_{H}}(\tau)(\mathbf e_{a}^{n}-\mathbf e_{b}^{n})=
-U_{A_{G}-A_{H}}(\tau)(\mathbf e_{a}^{n}-\mathbf e_{b}^{n})=\chi(\mathbf e_{a}^{n}-\mathbf e_{b}^{n}).
$$
Thus, (a) follows.

By (\ref{DCU}), if $G\ltimes H$ has Pair-LPST between $\mathbf e_{i}\otimes(\mathbf e_{a}^{n}-\mathbf e_{b}^{n})$ and $\mathbf e_{i}\otimes(\mathbf e_{c}^{n}-\mathbf e_{d}^{n})$ at time $\tau$, then
\begin{equation}\label{pairdouble-equation2}
\left | (\mathbf e_{a}^{n}-\mathbf e_{b}^{n})^{\top}(U_{A_{G}+A_{H}}(-\tau)+U_{A_{G}-A_{H}}(-\tau))(\mathbf e_{c}^{n}-\mathbf e_{d}^{n}) \right|^2=16.
\end{equation}
By Lemma \ref{Xlem}, Equation (\ref{pairdouble-equation2}) holds if and only if there exists a phase factor $\chi~(\left |\chi \right |=1 )$ satisfying that
$$
(\mathbf e_{a}^{n}-\mathbf e_{b}^{n})^{\top}U_{A_{G}+A_{H}}(\tau)(\mathbf e_{c}^{n}-\mathbf e_{d}^{n})=
(\mathbf e_{a}^{n}-\mathbf e_{b}^{n})^{\top}U_{A_{G}-A_{H}}(\tau)(\mathbf e_{c}^{n}-\mathbf e_{d}^{n})=2\chi,
$$
which implies that $A_{G}+A_{H}$ and $A_{G}-A_{H}$ have Pair-PST between $\mathbf e_{a}^{n}-\mathbf e_{b}^{n}$ and $\mathbf e_{c}^{n}-\mathbf e_{d}^{n}$ with the same phase factor $\chi$, leading to (b).

The proof of (c) is similar to that of (a). Hence, we omit the details here.
\qed
\end{proof}

%¶¨Àí~\ref{DC1}~¸ø³öÁËͼµÄË«¸²¸Ç´æÔÚÀ­ÆÕÀ­Ë¹ÍêÃÀ˫̬´«µÝµÄ³äÒªÌõ¼þ, µ«ÊÇÌõ¼þ²»¹»Ö±¹Û.

Theorem \ref{DC1} (a) states that if $G\ltimes H$ has Pair-LPST between $\mathbf e_{0}\otimes(\mathbf e_{a}^{n}-\mathbf e_{b}^{n})$ and $\mathbf e_{1}\otimes(\mathbf e_{a}^{n}-\mathbf e_{b}^{n})$, then $\mathbf e_{a}^{n}-\mathbf e_{b}^{n}$ is periodic with respect to $A_{G}+A_{H}$ and $A_{G}-A_{H}$ with phase factors $\chi$ and $-\chi$, respectively. That is,
$$
U_{A_{G}+A_{H}}(\tau)(\mathbf e_{a}^{n}-\mathbf e_{b}^{n})=\chi(\mathbf e_{a}^{n}-\mathbf e_{b}^{n}),~~
U_{A_{G}-A_{H}}(\tau)(\mathbf e_{a}^{n}-\mathbf e_{b}^{n})=-\chi(\mathbf e_{a}^{n}-\mathbf e_{b}^{n}).
$$
Assume that $A_{G}$ and $A_{H}$ commute. Then
$$
U_{A_{G}-A_{H}}(\tau)U_{A_{G}+A_{H}}(\tau)(\mathbf e_{a}^{n}-\mathbf e_{b}^{n})
=U_{A_{G}}(2\tau)(\mathbf e_{a}^{n}-\mathbf e_{b}^{n})
=-\chi^{2}(\mathbf e_{a}^{n}-\mathbf e_{b}^{n}),
$$
and
$$
U_{A_{G}+A_{H}}(\tau)U_{-A_{G}+A_{H}}(\tau)(\mathbf e_{a}^{n}-\mathbf e_{b}^{n})
=U_{A_{H}}(2\tau)(\mathbf e_{a}^{n}-\mathbf e_{b}^{n})
=-(\mathbf e_{a}^{n}-\mathbf e_{b}^{n}),
$$
which implies $\mathbf e_{a}^{n}-\mathbf e_{b}^{n}$ is periodic with respect to $A_{G}$ and $A_{H}$. Hence, we can construct new graphs having Pair-LPST by graphs having periodic pair states.

\begin{cor}\label{DC11}
Let $G$ be an $r_1$-regular graph and $H$ an $r_2$-regular graph, which have the same vertex set $V$. Suppose that $A_{G}$ and $A_{H}$ commute. Given vertices $a, b\in V$, assume that $\mathbf e_{a}^{n}-\mathbf e_{b}^{n}$ is periodic with respect to $A_G$ at time $\tau$, and $\mathbf e_{a}^{n}-\mathbf e_{b}^{n}$ is periodic with respect to $A_H$ at the same time $\tau$ with a phase factor $k\in\left \{ \pm \mathrm{i} \right \} $. Then $G\ltimes H$ has Pair-LPST between $\mathbf e_{0}\otimes(\mathbf e_{a}^{n}-\mathbf e_{b}^{n})$ and $\mathbf e_{1}\otimes(\mathbf e_{a}^{n}-\mathbf e_{b}^{n})$ at time $\tau$.
\end{cor}
\begin{proof}
Since $\mathbf e_{a}^{n}-\mathbf e_{b}^{n}$ is periodic with respect to $A_G$ at time $\tau$, there exists a phase factor $\chi~(\left |\chi \right |=1 )$ satisfying that
$$
U_{A_{G}}(\tau)(\mathbf e_{a}^{n}-\mathbf e_{b}^{n})
=\chi(\mathbf e_{a}^{n}-\mathbf e_{b}^{n}).
$$
Left-multiplying $U_{A_{G}}(-\tau)$ on both sides of the above equation, we have
\begin{align}\label{DCCG}
(\mathbf e_{a}^{n}-\mathbf e_{b}^{n})
=\chi U_{A_{G}}(-\tau)(\mathbf e_{a}^{n}-\mathbf e_{b}^{n}).
\end{align}
Consider the following cases.

\noindent\emph{Case~1.} $k=\mathrm{i}$. Note that $\mathbf e_{a}^{n}-\mathbf e_{b}^{n}$ is periodic with respect to $A_H$ at time $\tau$ with a phase factor $\mathrm{i}$. Similar to (\ref{DCCG}), we have
\begin{align}\label{DCCH}
(\mathbf e_{a}^{n}-\mathbf e_{b}^{n})
=\mathrm{i} U_{A_{H}}(-\tau)(\mathbf e_{a}^{n}-\mathbf e_{b}^{n}).
\end{align}
Combining (\ref{DCCG}) with (\ref{DCCH}), we get
\begin{align}\label{DCC1}
U_{A_{G}+A_{H}}(-\tau)(\mathbf e_{a}^{n}-\mathbf e_{b}^{n})\nonumber
&=U_{A_{G}}(-\tau)U_{A_{H}}(-\tau)(\mathbf e_{a}^{n}-\mathbf e_{b}^{n})\nonumber\\
&=\mathrm{i}^{-1}U_{A_{G}}(-\tau)(\mathbf e_{a}^{n}-\mathbf e_{b}^{n})\nonumber\\
&=-\mathrm{i}\chi^{-1}(\mathbf e_{a}^{n}-\mathbf e_{b}^{n}),
\end{align}
and
\begin{align}\label{DCC2}
U_{A_{G}-A_{H}}(-\tau)(\mathbf e_{a}^{n}-\mathbf e_{b}^{n})\nonumber
&=U_{A_{G}}(-\tau)U_{A_{H}}(\tau)(\mathbf e_{a}^{n}-\mathbf e_{b}^{n})\nonumber\\
&=\mathrm{i}U_{A_{G}}(-\tau)(\mathbf e_{a}^{n}-\mathbf e_{b}^{n})\nonumber\\
&=\mathrm{i}\chi^{-1}(\mathbf e_{a}^{n}-\mathbf e_{b}^{n}).
\end{align}
By \eqref{pairdouble-equation1}, (\ref{DCU}), (\ref{DCC1}) and (\ref{DCC2}), we have
$$
U_{L_{G\ltimes H}}(\tau)\mathbf e_{0}\otimes(\mathbf e_{a}^{n}-\mathbf e_{b}^{n})=-\mathrm{i}\chi ^{-1}\exp(-\mathrm{i}\tau(r_{1}+r_{2}))\mathbf e_{1}\otimes(\mathbf e_{a}^{n}-\mathbf e_{b}^{n}).
$$
Therefore, $G\ltimes H$ has Pair-LPST between $\mathbf e_{0}\otimes(\mathbf e_{a}^{n}-\mathbf e_{b}^{n})$ and $\mathbf e_{1}\otimes(\mathbf e_{a}^{n}-\mathbf e_{b}^{n})$.

\noindent\emph{Case~2.} $k=-\mathrm{i}$.  The proof is similar to that of Case 1, hence we omit the details here.
\qed
\end{proof}

\begin{example}\label{Example4}
{\em Let $K_{n}$ be a complete graph with $n~(n\ge 2)$ vertices. Note that the spectral decomposition of $A_{K_n}$ is $A_{K_n}=(n-1)E_{n-1}-E_{-1}$, where
\begin{align}\label{last}
E_{n-1}=\frac{1}{n}J_{n},~~E_{-1}=I-\frac{1}{n}J_{n}.
\end{align}
For any pair of vertices $\left\{a, b\right\}$ in $K_{n}$, by (\ref{ASpecDec2}) and (\ref{last}), we have
 $$
U_{A_{K_n}}\left(\frac{\pi}{2}\right)(\mathbf e_{a}^{n}-\mathbf e_{b}^{n})
=\exp\left(\mathrm{i}\frac{\pi}{2}\right)E_{-1}(\mathbf e_{a}^{n}-\mathbf e_{b}^{n})
=\mathrm{i}(\mathbf e_{a}^{n}-\mathbf e_{b}^{n}).
$$
 Therefore, $\mathbf e_{a}^{n}-\mathbf e_{b}^{n}$ is periodic with respect to $A_{K_n}$ at time $\frac{\pi}{2}$ with a phase factor $\mathrm{i}$. By Corollary \ref{DC11}, $K_{n}\ltimes K_{n}$ has Pair-LPST between $\mathbf e_{0}\otimes(\mathbf e_{a}^{n}-\mathbf e_{b}^{n})$ and $\mathbf e_{1}\otimes(\mathbf e_{a}^{n}-\mathbf e_{b}^{n})$ at time $\frac{\pi}{2}$.}
\end{example}

\section{Conclusion}
In this paper, we investigate the existence of Pair-LPST in tensor product and double cover of two regular graphs. We give necessary and sufficient conditions for the tensor product of two regular graphs to have Pair-LPST when one of the two regular graphs admits PST or Pair-LPST. We also give necessary and sufficient conditions for the double cover of two regular graphs to have Pair-LPST. By the results, we construct several families of graphs admitting Pair-LPST as examples, revealing the feasibility of constructing graphs having Pair-LPST by tensor products and double covers.

Future work could explore the existence of Pair-LPST in other graph operations, such as the lexicographic product, join, strong product and edge addition.

%Rohith and Sudheesh introduced the concept of \emph{fraction revival} (short for FR), which %is a generalization of PST based on the unitarity of transition matrix in \cite{Rohith}. Given %vertices $u$ and $v$ in $G$, FR occurs  between $u$ and $v$ at time $\tau$ if the transition %matrix associated with $A_{G}$ satisfies
%$$
% U_{A_{G}}(\tau)\mathbf e_{u}=\alpha \mathbf e_{u}-\beta\mathbf e_{v},
%$$
%where $\alpha$ and $\beta$ are complex numbers and %$\left|\alpha\right|^2+\left|\beta\right|^2=1$. The sufficient condition in Theorem %\ref{tensor3} relies on the unitarity of transition matrix. This motivates us to further %investigate the connection between Pair-LPST and FR, as well as characterize graph operations %having Pair-LPST based on some graphs with FR.

%\section{Acknowledgement}
%\jm{The authors would like to express their grateful thankfulness to the editor and the referee for their valuable comments and suggestions.}


\begin{thebibliography}{99}
\small{


\bibitem{Ack}
E. Ackelsberg, Z. Brehm, A. Chan, J. Mundinger, et al, Laplacian state transfer in coronas, Linear Algebra Appl. 506 (2016) 154--167.

\bibitem{Alvi}
 R. Alvir, S. Dever, B. Lovitz, J. Myer, et al, Perfect state transfer in Laplacian quantum walk, J. Algebraic Combin. 43 (4) (2016) 801--826.

\bibitem{Banks}
R. J. Banks, D. E. Browne, P. A. Warburton, Rapid quantum approaches for combinatorial optimisation inspired by optimal state-transfer, Quantum 8 (2024) 1253.

\bibitem{BOSE2}
S. Bose, Quantum communication through an unmodulated spin chain, Phys. Rev. Lett. 91 (20) (2003) 207901.

\bibitem{BOSE1}
S. Bose, A. Casaccino, S. Mancini, S. Serverini, Communication in XYZ all-to-all quantum networks with a missinng link, Int. J. Quantum Inf. 7 (4) (2009) 713--723

\bibitem{CAO}
X. Cao, Perfect edge state transfer on cubelike graphs, Quantum Inf. Process. 20 (9) (2021) 285.

\bibitem{CaoCL20}
X. Cao, B. Chen, S. Ling, Perfect state transfer on Cayley graphs over dihedral groups: the non-normal case,  Electron. J. Combin. 27 (2) (2020) \#P2.28.

\bibitem{CaoF21}
X. Cao, K. Feng, Perfect state transfer on Cayley graphs over dihedral groups, Linear Multilinear Algebra  69 (2) (2021) 343--360.

 \bibitem{CAO2}
X. Cao, J. Wan, Perfect edge state transfer on abelian Cayley graphs, Linear Algebra Appl. 653 (2022) 44--65.

\bibitem{QCh18}
Q. Chen, Edge state transfer, Master thesis, University of Waterloo, 2018.

\bibitem{ChG19}
Q. Chen, C. Godsil, Pair state transfer, Quantum Inf. Process. 19 (9) (2020) 321.

\bibitem{Cheung}
W. Cheung, C. Godsil, Perfect state transfer in cubelike graphs, Linear Algebra Appl. 435 (10) (2011) 2468--2474.

\bibitem{chris1}
M. Christandl, N. Datta, A. Ekert, A. Landahl, Perfect state transfer in quantum spin networks, Phys. Rev. Lett. 92 (18) (2004) 187902.

\bibitem{Coutinho14}
G. Coutinho, Quantum State Transfer in Graphs, PhD thesis, University of Waterloo, 2014.


\bibitem{Cohgo16}
G. Coutinho, C. Godsil, Perfect state transfer in products and covers of graphs, Linear Multilinear Algebra 64 (2) (2016) 235--246.

\bibitem{distance}
G. Coutinho, C. Godsil, K. Guo, F. Vanhove, Perfect state transfer on distance-regular graphs and association schemes, Linear Algebra Appl. 478 (2015) 108--130.

\bibitem{Coutinho11}
G. Coutinho, H. Liu, No Laplacian perfect state transfer in trees, SIAM J. Discrete Math. 29 (4) (2015) 2179--2188.

\bibitem{JM}
M. Jiang, X. Liu, J. Wang, Laplacian pair state transfer in Q-graph, Discrete Appl. Math. 375 (2025) 239--258.

\bibitem{Kim}
S. Kim, H. Monterde, B. Ahmadi, A. Chan, et al, A generalization of quantum pair state transfer, Quantum Inf Process 23 369 (2024).

\bibitem{Kirk}
S. Kirkland, S. Severini, Spin-system dynamics and fault detection in threshold networks, Phys. Rev. A 83 (1) (2011) 012310.

\bibitem{LiY2}
 Y. Li, X. Liu, S. Zhang, Laplacian state transfer in Q-graph, Appl. Math. Comput. 384 (2020) 125370.

\bibitem{LiLZZ21}
Y. Li, X. Liu, S. Zhang, S. Zhou, Perfect state transfer in NEPS of complete graphs, Discrete Appl. Math. 289 (2021) 98--114.

\bibitem{LiuD}
D. Liu, X. Liu, No Laplacian perfect state transfer in total graphs, Discrete Math. 346 (9) (2023) 113529.

\bibitem{LCXC21}
G. Luo, X. Cao, G. Xu, Y. Cheng, Cayley graphs of dihedral groups having perfect edge state transfer, Linear Multilinear Algebra  70 (20) (2022) 5957--5972.

\bibitem{Main}
D. Main, P. Drmota, D. P. Nadlinger, E. M. Ainley et al, Distributed quantum computing across an optical network link, Nature 638 (2025) 383--388.

\bibitem{Pet}
M. D. Petkovi\'{c}, M. Ba\v{s}i\'{c}, Further results on the perfect state transfer in integral circulant graphs, Comput. Math. Appl. 61 (2) (2011) 300--312.

%\bibitem{Rohith}
%M. Rohith, C. Sudheesh, Visualizing revivals and fractional revivals in a Kerr medium using an %optical tomogram, Phys. Rev. A 92 (2015) 053828.

\bibitem{Tan19}
Y. Tan, K. Feng, X. Cao, Perfect state transfer on abelian Cayley graphs, Linear Algebra Appl. 563 (2019) 331--352.

\bibitem{Tao}
Y. Tao, W. Wang, Cayley graphs of semi-dihedral groups having perfect edge state transfer, Acta Sci. Natur. Univ. Sunyatseni 64 (4) (2025) 147--155.

\bibitem{Tm19}
I. Thongsomnuk, Y. Meemark, Perfect state transfer in unitary Cayley graphs and gcd-graphs, Linear Multilinear Algebra  67 (1) (2019) 39--50.

\bibitem{WJ}
J. Wang, X. Liu, Laplacian state transfer in edge complemented coronas, Discrete Appl. Math. 293 (2021) 1--14.

\bibitem{liu2023}
W. Wang, X. Liu, J. Wang, Laplacian pair state transfer in vertex coronas, Linear Multilinear Algebra 71 (14) (2023) 2282--2297.

\bibitem{Xu}
J. Xu, F. Mei, Y. Zhu, Spectrum-based shortcut method for topological systems, Phys. Rev. A 110 (2024) 032431.

\bibitem{Yin}
J. Yin, Y. Cao, Y. Li, S. Liao, et al, Satellite-based entanglement distribution over 1200 kilometers, Science 356 (2017) 1140--1144.

\bibitem{ZHU}
X. Zhu, T. Tu, A. Guo, Z. Zhou, et al, Spin-photon module for scalable network architecture in quantum dots, Sci. Rep. 10 (2020) 5063.

}
\end{thebibliography}
\end{document}